\documentclass[a4paper,11pt,fleqn]{amsart}

\usepackage[fontsize=11pt]{scrextend}
\usepackage{tikz}
\usepackage{tkz-graph}
\usetikzlibrary{positioning}
\usetikzlibrary{arrows.meta}
\usepackage{stackengine}
\usepackage{amsmath,amsfonts}
\usepackage{amssymb,amsthm,mathtools}
\usepackage{mathrsfs,nccmath,amssymb,amsthm,mathtools}
\usepackage{filecontents}
\usepackage{hyperref}
\usepackage{enumitem}
\usepackage[all]{xy}
\xyoption{line}

\usepackage{color}

\usepackage[normalem]{ulem}
\renewcommand{\le}{\leq}
\renewcommand{\ge}{\geq}

\newcommand{\undersim}[1]{\stackMath\stackunder[-0.2pt]{\hspace{7pt}\sim\hspace{7pt}}{\scalebox{0.6}[0.6]{$#1$}}}
\newcommand{\simminl}{\undersim{\phantom{\ }\mathrm{min},l}}

\newcommand{\minus}{\raisebox{0.2ex}{\scalebox{0.6}[0.6]{{\boldmath$-$}}}}
\newcommand{\plus}{\raisebox{0.2ex}{\scalebox{0.6}[0.6]{{\boldmath$+$}}}}
\newcommand{\sminus}{\scalebox{0.5}[0.5]{{\boldmath$-$}}}
\newcommand{\splus}{\scalebox{0.5}[0.5]{{\boldmath$+$}}}
\newcommand{\alignb}{\[\begin{aligned} }
\newcommand{\alignn}{ \end{aligned}\]}

\DeclareMathOperator{\lcm}{lcm}
\DeclareMathOperator{\gap}{gap}

\newcommand\mL{L\kern-0.08cm\char39}

\newcommand{\mywith}{\text{ with }}
\newcommand{\myforsome}{\text{ for some }}

\newcommand{\myand}{\text{ and }}

\newcommand{\seb}{\{\,}
\newcommand{\sen}{\,\}}

\newcommand{\kuu}{\emptyset}
\newcommand{\nekuu}{\neq \kuu}
\newcommand{\iskuu}{= \kuu}

\newcommand{\ew}{\epsilon}
\newcommand{\fai}{\varphi}

\newcommand{\Lett}{\operatorname{Lett}}

\newcommand{\pdirectional}{\raise0.05em\hbox{$+$}directional}
\newcommand{\pdirectionality}{\raise0.05em\hbox{$+$}directionality}
\newcommand{\pdirectionalitys}{\raise0.05em\hbox{$+$}directionality }
\newcommand{\pdirectionals}{\raise0.05em\hbox{$+$}directional }
\newcommand{\mdirectional}{\raise0.05em\hbox{$-$}directional}
\newcommand{\mdirectionality}{\raise0.05em\hbox{$-$}directionality}
\newcommand{\mdirectionalitys}{\raise0.05em\hbox{$-$}directionality }
\newcommand{\mdirectionals}{\raise0.05em\hbox{$-$}directional }

\newcommand{\Z}{\mathbb{Z}}

\newcommand{\bi}{\in \Z}

\newcommand{\bpi}{\ge 1}

\newcommand{\sL}{\mathscr{L}}
\newcommand{\sM}{\mathscr{M}}

\newcommand{\sS}{\mathscr{S}}

\newcommand{\centb}{\begin{center}}
\newcommand{\centn}{\end{center}}
\newcommand{\enumb}{\begin{enumerate}}
\newcommand{\enumn}{\end{enumerate}}
\newcommand{\itemb}{\begin{itemize}}
\newcommand{\itemn}{\end{itemize}}
\numberwithin{equation}{section}
\usepackage[left=30mm,right=30mm,top=25mm,bottom=25mm]{geometry}
\usepackage{cleveref}
\setlist[enumerate,1]{label=(\alph*),ref=(\alph*)}
\setlist[enumerate,2]{label=(\arabic*),ref=(\alph{enumi}-\arabic{enumii})}
\setlist[enumerate,3]{label=(\Alph*),ref=(\roman{enumi}-\alph{enumii}-\Alph*)}
\setlist[enumerate,4]{label=(\arabic*),ref=(\roman{enumi}-\alph{enumii}-\Alph{enumiii}-\arabic*)}
\newlist{enumrm}{enumerate}{1}
\setlist[enumrm,1]{label={\rm (\roman*)},ref=(\roman*)}
\newcommand{\enumrmb}{\begin{enumrm}}
\newcommand{\enumrmn}{\end{enumrm}}
\newlist{enumsec}{enumerate}{1}
\crefalias{enumseci}{enumi} 
\setlist[enumsec,1]{label={\rm (\thesection.\alph*)},ref=(\thesection.\alph*)}
\newcommand{\enumsecb}{\begin{enumsec}[resume]}
\newcommand{\enumsecn}{\end{enumsec}}
\newlist{deepenum}{enumerate}{1}
\setlist[deepenum,1]{label=($1$:\alph*),ref=(1:\alph*)}
\newlist{Lminusoneenum}{enumerate}{1}
\setlist[Lminusoneenum,1]{label=($-1$:\alph*),ref=($-1$:\alph*)}
\newlist{Ldeepenum}{enumerate}{1}
\setlist[Ldeepenum,1]{label=($2$:\alph*),ref=($2$:\alph*)}
\numberwithin{equation}{section}
\newtheorem{thm}[equation]{Theorem}
\newtheorem{lem}[equation]{Lemma}
\newtheorem{prop}[equation]{Proposition}
\newtheorem{cor}[equation]{Corollary}
\newtheorem{mainthm}{Theorem}

\theoremstyle{definition}
\newtheorem{defn}[equation]{Definition}

\newtheorem{example}[equation]{Example}

\theoremstyle{remark}

\newtheorem{nota}[equation]{Notation}
\newtheorem{rem}[equation]{Remark}

\crefname{sec}{\S}{\S\S}
\crefname{mainthm}{Theorem}{Theorems}
\crefname{thm}{Theorem}{Theorems}
\crefname{lem}{Lemma}{Lemmas}
\crefname{prop}{Proposition}{Propositions}
\crefname{cor}{Corollary}{Corollaries}
\crefname{defn}{Definition}{Definitions}
\crefname{conj}{Conjecture}{Conjectures}
\crefname{example}{Example}{Examples}
\crefname{nota}{Notation}{Notations}
\crefname{rem}{Remark}{Remarks}
\crefname{note}{Note}{Notes}
\crefname{case}{Case}{Cases}
\crefname{figure}{Figure}{Figures}
\crefname{section}{\S}{\S\S}
\crefname{enumi}{}{}
\crefname{enumii}{}{}
\crefname{equation}{}{}

\newcommand{\abs}[1]{\lvert#1\rvert}

\begin{document}
  
\title[A Simple Approach to substitution]{A simple approach to substitution\\
 minimal subshifts}

\author{Takashi Shimomura}
\address{Nagoya University of Economics, Uchikubo 61-1, Inuyama 484-8504, Japan}
\email{tkshimo@nagoya-ku.ac.jp}

\thanks{}

\subjclass[2010]{Primary 37B10, 54H20. }

\keywords{substitution, subshift, dynamical system}

\date{\today}



\begin{abstract}
In the study of substitution minimal subshifts,
 some complicated trivialities have hindered
 simple and general approaches.
Recently, Maloney and Rust introduced the term ``tame,''
 simplifying the study.
We introduce another term ``$l$-primitive'' for the substitutions and
 show that the combination of these two conditions can characterize
 the minimality of substitution subshifts.
We shall show that
 all substitution minimal subshifts can be generated by 
 substitutions that satisfy both conditions;
 conversely, all substitutions that satisfy the two conditions always
 generate minimal subshifts.
As an application, we show that the result by Damanik and Lenz that
 an admissible substitution subshift is minimal
 if and only if it is linearly repetitive is valid
 for all substitution subshifts.
The above set of conditions can be checked by finite calculations (algorithms).
\end{abstract}

\maketitle

\section{Introduction}\label{sec:introduction}
The study of minimal substitution subshifts has an extensive history
 (see, for example,
 \cite{Fogg_2002SubstitutionsInDynamicsArithmeticsAndCombinatorics}).
Historically, minimal substitutions are discussed in the context of
 primitive substitutions (see, for example,
\cite{Gottschalk1963SubstitutionMinimalSets,Martin_1971SubstitutionMinimalFlows,CovenKeane1971StructureOfSubstitutionMinimalSets,Martin1973MinimalFlowsArisingFromSubstitutionsOfNonConstantLength,Queffelec1987SubstitutionDynamicalSystemsSpectralAnalyss,Mosse1992PuissanceDeMotsEtReconnaissabiliteDesPointsFixes}).
However, many minimal substitution subshifts are generated by nonprimitive
 substitutions.
In \cite{de2002SingularContinuousSpectrumForAClssOfNonprimitiveSubstitutionSchrodingerOperators}, de Oliveira and Lima investigated a nonprimitive and yet minimal substitution subshift on two symbols.
The relation between the primitivity and the minimality is very subtle in 
 the theory of substitution subshifts.
Additionally, the minimality is an abstract notion
 in the theory of topological dynamical systems.
In this paper, we show that, in the theory of substitution subshifts,
 there exists a concrete characterization of the minimality
 (see \cref{thm:main2}).
However,
 there remains a profound question of
 whether or not all of the substitution minimal subshifts
 can be generated by primitive substitutions.
In \cite[\S 6.2]{Durand1998ACharacterizationOfSubstitutiveSequencesUsingReturnWords}, Durand first argued this.
Later, Durand \cite[Theorem 3]{Durand_2013DecidabilityOfUniformRecurrenceOfMorphicSequences} has given, in the extended framework of morphic sequences,
 a partial answer to this question, stating that 
 uniformly recurrent morphic sequences are primitive substitutive sequences.
Owing to this result, after a sketch of the proof in the one-sided infinite case 
 that had been made in \cite[\S 6.2]{Durand1998ACharacterizationOfSubstitutiveSequencesUsingReturnWords}, Maloney and Rust \cite[Theorem 2.1]{Maloney_Rust_2016BeyondPrimitivityForOneDimensionalSubstitutionSubshiftsAndTilingSpaces},
 have given a complete proof of the fact that
 for a minimal substitution $\fai$
 with the nonempty minimal subshift $X_{\fai}$,
  there exist an alphabet $Z$ and a primitive substitution $\theta$ on $Z$
  such that $X_{\theta}$ is topologically conjugate to $X_{\fai}$.
However, the conjugating maps change the symbolic sequences
 drastically.
Therefore, even now, we \textbf{cannot} study
 all substitution minimal subshifts
 as primitive substitution subshifts. 

Damanik and Lenz in \cite{Damanik2006SubstitutionDynSysCharactOfLinearRepetAndApplications} extended some results from the framework of primitive substitutions to the framework of substitution minimal subshifts.
In \cite[Theorem 1]{Damanik2006SubstitutionDynSysCharactOfLinearRepetAndApplications},
 they had shown that for admissible substitution subshifts,
 the minimality is equivalent to the linear repetitivity. 
We think that the characterization of substitutions
 that generate substitution minimal subshifts
 is becoming important.
The outline of our answer is that
 every substitution minimal subshift can be generated
 by a substitution that satisfies the combination of two conditions
 that are stated below;
 conversely, substitutions with a set of
 two conditions always generate
 substitution minimal subshifts (see \cref{thm:main2}).
One of the two conditions is the tameness that had been discussed in \cite{Maloney_Rust_2016BeyondPrimitivityForOneDimensionalSubstitutionSubshiftsAndTilingSpaces}.
They stated that in the presence of tameness, most of the
 possible pathological behaviors of nonminimal substitutions cannot occur.
Another condition we propose is the $l$-primitivity
 that is defined later in this section.
Although both conditions have been discussed in many works
 perhaps without naming them,
 there exist no clear statements that the combination of the two conditions,
 tameness and $l$-primitivity, ``are equal'' to the minimality.
In fact, we can say that
 for every substitution minimal subshift $X$,
 there exists a tame and $l$-primitive substitution that generates $X$
 (see \cref{thm:main2}).
We think that some complicated trivialities
 that have hindered a simple and general approach to substitution minimal
 subshifts have been removed.
As an example, we can show that 
 the coincidence of the minimality and linear repetitivity that had been 
 shown in the case of admissible substitution subshifts by
 Damanik and Lenz in \cite[Theorem 1]{Damanik2006SubstitutionDynSysCharactOfLinearRepetAndApplications} is valid for all substitution subshifts
 (see \cref{cor:DL_my_way}).

Hereon, we discuss concretely.
Let $\Z$ be the set of all integers and $A$ be
 a nonempty finite set that is called an \textit{alphabet}.
We exclude the case in which $A$ consists of a single element.
Each element $a \in A$ is called a \textit{letter}.
A finite sequence $w = u_1 u_2 \dotsb u_k$ $(u_i \in A, i=1,2,\dotsc,k)$
 is called a \textit{word}, and its \textit{length} $k$ is denoted as $\abs{w}$.
Let $A^+$ be the set of all words with a positive finite length.
In addition, we define an \textit{empty word} $\ew$
 with length $\abs{\ew} = 0$.
We define the set $A^* := A^+ \cup \seb \ew \sen$.
For $u, v \in A^*$, we define the concatenation of words as $u v$.
It follows that $\abs{a} = 1$ for all $a \in A$, and 
 $\abs{uv} = \abs{u}+ \abs{v}$ for all $u, v \in A^*$.
Suppose that there exist words $u,u',v,w \in A^*$ that satisfy
 $w = u v u'$.
Then, $v$ is said to be a \textit{subword} or a \textit{factor} of $w$.
Let $\sigma : A \to A^+$ be a map.
Then, we can extend the map $\sigma$ as $\sigma : A^* \to A^*$, i.e.,
 $\sigma(u_1 u_2, \dotsc, u_k) = \sigma(u_1) \sigma(u_2) \dotsc \sigma(u_k)$
 for $u_i \in A$ $(1 \le i \le k)$ and $\sigma(\ew) = \ew$.
For a map $\sigma : A \to A^+$, we define $A_l := \seb a \in A \mid \lim_{n \to \infty}\abs{\sigma^n(a)} = \infty \sen$ and $A_s := A \setminus A_l$.
We define the set $\sL(\sigma) \subseteq A^*$ that consists of all words $v$
 that are subwords of $\sigma^n(a)$ for some $a \in A$ and $n \bpi$.
We say that $\sL(\sigma)$ is the \textit{language} of $\sigma$.
Let $A^{\Z}$ be the set of all two-sided sequences of the letters.
For each $x \in A^{\Z}$ and every finite interval $[s,t]$
 of integers with $s < t$ $(s,t \in \Z)$,
 we consider a finite \textit{block} or a word as
 $x[s,t] := x(s) x(s+1) \dotsb x(t) \in A^+$.
We denote $\sL(x) := \seb x[s,t] \mid s < t,\ s,t \bi \sen$.
For each $v \in \sL(x)$, we also write $v \prec x$.
We assume that $\ew \in \sL(x)$.
For a subset $X$ of $A^{\Z}$,
 we denote $\sL(X) := \bigcup_{x \in X}\sL(x)$.
We define the \textit{two-sided full shift over} $A$ as
 the pair $(A^{\Z},T)$, where $T : A^{\Z} \to A^{\Z}$ is the left shift.
For a subset $\Lambda \subseteq A^{\Z}$, we say that $(\Lambda,T|_\Lambda)$
 is a \textit{subshift} if 
 it is closed and satisfies $T(\Lambda) = \Lambda$.
We write $(\Lambda,T)$ instead of $(\Lambda,T|_\Lambda)$
We define a subshift $(X_{\sigma},T)$ of $(A^{\Z},T)$
 as
\begin{ceqn}\begin{align*}
X_{\sigma}
 := \seb x \in A^{\Z}
 \mid x[s,t] \in \sL(\sigma) \text{ for all integers } s < t \sen.
\end{align*}‎\end{ceqn}

For a map $\sigma : A \to A^+$, we say that a letter $b \in A$ is
 \textit{isolated} if $b \in A_s$ and
 $b$ does not appear in any $\sigma^n(a)$ $(a \in A_l, n \bpi)$.
We denote the set of all isolated letters as $A_{\text{iso}}$.
We say that $\sigma$ is a \textit{substitution}
 if $A_l \nekuu$ and $A_{\text{iso}} \iskuu$.
We have excluded the case in which $A_l \iskuu$.
This is a necessary exclusion if we need infinitely many words with the form
 $\sigma^n(a)$ $(a \in A, n \bpi)$.
If $A_l \iskuu$, then it follows that $X_{\sigma} \iskuu$. 
We note that we have imposed the new condition $A_{\text{iso}} \iskuu$.
This removes a small portion of the trivial hindrances
 in starting the study of general
 substitution minimal subshifts.
This is not a strong restriction if we consider
 only the phenomena of the words $\sigma^n(a)$ $(a \in A_l, n \bpi)$.
Letters in $A_{\text{iso}}$ do not appear
 in the words $\sigma^n(a)$ $(a \in A_l, n \bpi)$
 nor in the elements of $X_{\sigma}$.
Let $\sigma : A \to A^+$ be a substitution.
Then, it always follows that
 $\sL(X_{\sigma}) \subseteq \sL(\sigma)$.
However, the converse inclusion does not hold always
 (see \cref{example:non-admissible}).
The substitution $\sigma$ that satisfies $\sL(X_{\sigma}) = \sL(\sigma)$
 is said to be
 \textit{admissible}
 (see \cite{Maloney_Rust_2016BeyondPrimitivityForOneDimensionalSubstitutionSubshiftsAndTilingSpaces}).
We do not exclude the case in which $A_l$ consists of a single element.
In some contexts of important studies, it is often assumed that $A = A_l$ or
 $A_s \iskuu$.
However, after \cite{DURAND_1999SubstDynSysBratteliDiagDimGroup,Lagarias_Pleasants_2003},
 Damanik and Lenz \cite{Damanik2006SubstitutionDynSysCharactOfLinearRepetAndApplications} raised the importance of the study of general substitution minimal
 subshifts without the assumption that $A_s \iskuu$.

Let $\sigma : A \to A^+$ be a substitution.
We consider the next condition:
\begin{align}\label{item:l-primitive}
 \text{There exists an } n \bpi \text{ such that for all } a, b \in A_l,
 a \text{ is a letter in } \sigma^n(b).
\end{align}
In the case that $A_s \iskuu$, the substitution $\sigma$ that satisfies
 condition \cref{item:l-primitive} has been called \textit{primitive}
 and has been studied a lot (see, for example, the aforementioned studies).
In addition, without the assumption that $A_s \iskuu$, 
 this condition has been studied extensively,
 especially in \cite{Durand1998ACharacterizationOfSubstitutiveSequencesUsingReturnWords,Durand_2013DecidabilityOfUniformRecurrenceOfMorphicSequences}.
However, the fact that this condition with another condition
 that is stated later in this section
 can characterize the minimality of
 all substitution subshifts has not been clearly stated.
We say that a substitution $\sigma$ is \textit{$l$-primitive}
 if it satisfies condition \cref{item:l-primitive}.

Another condition that we introduce is due to Maloney and Rust \cite{Maloney_Rust_2016BeyondPrimitivityForOneDimensionalSubstitutionSubshiftsAndTilingSpaces}.
From their work, we use the following definitions.
A letter $a \in A_l$ is \textit{left-isolated} if
 there exists an $n \bpi$ such that
 $\sigma^n(a) =  l_s(a)\ a\ w_{\plus}(a)$ with
 $l_s(a) \in {A_s}^+$ and $w_{\plus}(a) \in A^*$.
 Further, a letter $a \in A_l$ is \textit{right-isolated} if
 there exists an $n \bpi$ such that
 $\sigma^n(a) =  w_{\minus}(a)\ a\ r_s(a)$ with
 $r_s(a) \in {A_s}^+$ and $w_{\minus}(a) \in A^*$.
We can consider the extended $\sigma : A^{\Z} \to A^{\Z}$ on infinite words
 such that $\sigma(x)$ is defined as
 $\cdots \sigma(x(-1)). \sigma(x(0)) \sigma(x(1)) \cdots $,
 in which the nearest right side of the dot ``$\ .\ $''
 is the position of $(\sigma(x))(0)$.
We use the term ``isolated.'' To explain the reason, think of an occurrence of a left-isolated letter $a \in A_l$ that appears in an $x \in A^{\Z}$,
and apply $\sigma^{in}(x)$ $(i = 1,2,\dotsc)$.
This occurrence of $a$ will be isolated from all of the occurrences of
 the letters in $A_l$
 on the left side of the occurrence of $a$.
We propose the following condition:
\begin{align}\label{item:not-isolated}
 \text{No letter in } A_l \text{ is left-isolated, and no letter in } A_l \text{ is right-isolated}.
\end{align}
This condition has also been discussed in previous studies.
Pansiot discussed this condition in \cite{Pansiot1984ComplexiteDesFacteursDesMotsInfinisEngendresParMorphismesIteres}.
Maloney and Rust named this as ``tame'' in
\cite{Maloney_Rust_2016BeyondPrimitivityForOneDimensionalSubstitutionSubshiftsAndTilingSpaces}, in contrast to the notion of wildness
 in their study of substitution subshifts.
A substitution is \textit{tame} if it satisfies
 condition \cref{item:not-isolated}.
Our definition of tameness is equivalent to that by Maloney and Rust.
We obtain
\begin{mainthm}\label{thm:main}
Let $\sigma : A \to A^+$ be a substitution.
Suppose that $\sigma$ is tame and $l$-primitive.
Then, $(X_{\sigma},T)$ is minimal.
Conversely,
 suppose that $(X_{\sigma},T)$ is minimal and is not a single periodic orbit.
Then, the substitution $\sigma$ is tame, and
 there exists a subalphabet $B \subseteq A$
 and the restriction $\sigma|_B : B \to B^+$
 such that $\sigma|_B$ is tame and $l$-primitive
 and $(X_{\sigma},T) = (X_{\sigma|_B},T)$.
\end{mainthm}

The first half of the theorem is well-known, and the last half
 has an easy proof with some knowledge of the theory of nonnegative matrices
 (see, for example, \cite[Proposition 2.65]{Rigo2014FormalLanguagesAutomataAndNumerationSystems1}).
However, in this paper, we propose other self-contained proofs for the
 ease of the reader.
We only assume \cref{prop:minimal-subshifts}
 as a basic fact about minimal subshifts.
From the above theorem, we can obtain the following result:

\begin{mainthm}\label{thm:main2}
Let $\sM$ be the class of all substitution minimal subshifts.
Let $\sM'$ be the class of all $(X_{\sigma},\sigma)$
 such that $\sigma$ is a substitution that is tame and $l$-primitive.
Then, it follows that $\sM = \sM'$.
\end{mainthm}

Thus, if we need to consider substitution minimal subshifts,
 we can always assume that they are concretely generated
 by tame and $l$-primitive substitutions.
Note that we have not assumed any conditions for the substitutions
 other than $A_l \nekuu$ and $A_{\text{iso}} \iskuu$.
Our condition for minimality is the combination of
 tameness and $l$-primitivity.
It is easy to see that these conditions can be checked by finite calculations (see \cref{rem:check-not-isolated,rem:calculation-primitive-Al}).  
We note that,
 in \cite[Theorem 1]{Durand_2013DecidabilityOfUniformRecurrenceOfMorphicSequences},
 Durand had already shown that uniform recurrence is decidable
 in the extended framework of morphic sequences.

As byproducts of our result, some trivialities that have hindered
 a general approach in the study of substitution minimal subshifts
 have been removed.
In many works concerning substitution minimal subshifts,
 it is often assumed that there exists a letter $a$ such that 
 $\sigma(a)$ begins from $a$.
However, it is an easy task to show that the condition is
 derived from the assumption that the substitution $\sigma$ is tame
 if we consider $\sigma^p$ for some $p \bpi$
 (see \cref{rem:tame-implies-head-letter}).
Thus, in consideration of minimal (in fact, tame) substitution subshifts,
 it has been clarified that
 the condition is not too restrictive.
In addition,
 on the assumption that the substitution is tame and $l$-primitive,
 we can easily conclude that $\sL(X_{\sigma}) = \sL(\sigma)$
 (see \cref{prop:admissible}).
Thus, we can show that \cite[Theorem 1]{Damanik2006SubstitutionDynSysCharactOfLinearRepetAndApplications} is valid for all substitution subshifts
 (see \cref{cor:DL_my_way}).

\section{Preliminaries}\label{sec:preliminaries}
In this section,
 we propose some additional notation and definitions.
Let $A$ be an alphabet and $u,v \in A^+$.
If $u$ is a factor of $v$, then we write $u \prec v$.
Note that $\ew \prec v$ for every word $v$.
For $u \in A^*$, we consider
 a \textit{language} $\sL(u) := \seb v \mid v \prec u \sen$.
For $u, v \in A^+$ with $\abs{u} \le \abs{v}$,
 we denote $\abs{v}_u$ as the number of occurrences of $u$ in $v$.
For example, we obtain $\abs{1101010}_{1010} = 2$.
In particular, we obtain $\abs{u} = \sum_{a \in A}\abs{u}_a$.
For $u \in A^+$, we denote
 $\Lett(u) := \sL(u) \cap A = \seb a \in A \mid a \prec u \sen$.
For $x \in A^{\Z}$, we denote
 $\Lett(x) := \seb a \in A \mid a = x(i) \myforsome i \bi \sen$.
For a subset $X \subseteq A^{\Z}$, we denote
 $\Lett(X) := \bigcup_{x \in X}\Lett(x)$.

After Durand
 \cite{Damanik2006SubstitutionDynSysCharactOfLinearRepetAndApplications},
 Damanik and Lenz
 in \cite{Damanik2006SubstitutionDynSysCharactOfLinearRepetAndApplications}
 studied linear repetitivity.

\begin{defn}
A subshift $(\Lambda,T)$
 is \textit{linearly repetitive} if 
 there exists a constant $C_{\rm LR} > 0$ 
 such that for all $v,w \in \sL(\Lambda)$ with $\abs{w} \ge C_{\rm LR}\abs{v}$,
 $\abs{w}_v \ge 1$ is satisfied.
\end{defn}

\begin{defn}
Let $W \subset A^+$.
We denote 
\[\sL(W) := \seb v \mid v \prec w \text{ for some } w \in W \sen.\]
\end{defn}

\begin{lem}\label{lem:infinite-words-implies-non-empty}
Let $W \subset A^+$ be an infinite set.
Then, there exists an $x \in A^{\Z}$ such that $\sL(x) \subseteq \sL(W)$.
\end{lem}
\begin{proof}
Without loss of generality, we remove the elements of $A \cup A^2$ from $W$.
If $\abs{w}$ is even, then we cut the first (or the last) letter from $w$.
In this way, we obtain a new set $W'$.
For each $w \in W'$, we define $l_w := (\abs{w}-1)/2$.
For each $w \in W'$, take an arbitrary $x_w \in A^{\Z}$ such that
 $x_w[-l_w,l_w] = w$.
Because $A$ is a finite set, we can find
 a sequence $w(i) \in W'$ $(i \bpi)$ such that
 $l_{w(1)} < l_{w(2)} < \dotsb$ and
 $x_{w(i)}[-l_{w(i)},l_{w(i)}] = x_{w(i+1)}[-l_{w(i)},l_{w(i)}]$
 for every $i \bpi$.
Thus, there exists an $x \in A^{\Z}$ such that
 $x_{w(i)} \to x$ as $i \to \infty$.
Now, it is clear that $\sL(x) \subseteq \sL(W') \subseteq \sL(W)$.
This completes the proof.
\end{proof}

Here, we introduce the notion of minimality.
Let $(\Lambda,T)$ be a subshift and $x \in \Lambda$.
Let $v \prec x$ be a finite block.
We say that $v$ \textit{appears in $x$ in bounded gaps} if
 there exists an $L_{x,v} > 0$ such that
 each finite block $w \prec x$ with $\abs{w} \ge L_{x,v}$
 satisfies $v \prec w$.
A subshift $(\Lambda,T)$ is \textit{minimal} if it does not
 contain any nonempty proper subshifts.
The next proposition is well-known (see, for example, \cite[\S 2]{Damanik2006SubstitutionDynSysCharactOfLinearRepetAndApplications} and \cite[\S 13.7]{Lind_1995AnIntroSymbDynSys}):
\begin{prop}\label{prop:minimal-subshifts}
A subshift $(\Lambda,T)$ is minimal
 if and only if every $v \in \sL(\Lambda)$ appears in bounded gaps
 in each $x \in \Lambda$.
In particular, linearly repetitive subshifts are minimal.
Furthermore, a subshift $(\Lambda,T)$ is minimal
 if and only if there exists an $x \in \Lambda$ such that
 each finite block $w \prec x$ appears in bounded gaps
 and $\Lambda$ is the orbit closure of $x$.
\end{prop}
\begin{proof}
Because this fact is well-known, we omit the proof here (see the
 aforementioned study).
\end{proof}

In this section, we give an example of a substitution that is not admissible.
\begin{example}\label{example:non-admissible}
Let $A = \seb 0, 1 \sen$ be an alphabet and define $\sigma : A \to A^+$
 by $\sigma(0) = 0$ and $\sigma(1) = 10$.
Then, $\sigma$ is a substitution in this paper,
 and $X_{\sigma}$ is a single fixed point $\cdots 000.000 \cdots$.
However, it follows that $100 \cdots 0 \in \sL(\sigma)$ with an arbitrary length.
\end{example}

\section{Substitutions}
\label{sec:main}
In this section, we give all of the proofs of our main results as well as
 additional basic definitions and notation.
This section seems to be lengthy for its content because we took 
 a self-contained approach.
If the reader knows the theory of nonnegative matrices,
 it may be challenging to give a short proof for this.
We tried to refer to some of the leading results that is obtained from this proof.
However, we have to apologize that our discussion is not sufficient.

Let $A$ be an alphabet and $\sigma : A \to A^+$ be a substitution.
It is easy to see that for every $a \in A_l$ and every $n \bpi$,
 it follows that $\Lett(\sigma^n(a)) \cap A_l \nekuu$.
It is evident that for each $a \in A_s$,
 the sequence $\sigma(a), \sigma^2(a), \sigma^3(a), \dotsc$
 is eventually periodic.
Thus, for each $a \in A_s$,
 we can take an arbitrarily large $n_s(a) \bpi$ and period $p_s(a)$
 such that if we set $w_s(a) := \sigma^{n_s(a)}(a)$, then
 $\sigma^{p_s(a)}(w_s(a)) = w_s(a)$.
Further, we obtain $\abs{\sigma^i(w_s(a))} = \abs{w_s(a)}$ for all $i \bpi$.

\begin{defn}\label{defn:As-const}
We define $n_s := \max_{a \in A_s}n_s(a)$,
 $p_s := \lcm \seb p_s(a) \mid a \in A_s \sen$, and
 $k_s := \max \seb \abs{w_s(a)} \mid a \in A_s \sen$.
\end{defn}

Damanik and Lenz in
\cite{Damanik2006SubstitutionDynSysCharactOfLinearRepetAndApplications}
 considered substitution dynamical systems $(X_{\sigma},T)$
 with the following assumptions:
\begin{itemize}[label={},topsep=12pt]
\setlength\itemsep{9pt}
\item
\makebox[0pt][l]{
\begin{minipage}[b]{\textwidth}
\begin{align}\label{item:infinite-char}
 \text{there exists an element } e \in A_l,
\end{align}
\end{minipage}}
\item
\makebox[0pt][l]{
\begin{minipage}[b]{\textwidth}
\begin{align}\label{item:all-char}
 A = \bigcup_{n \bpi} \Lett( \sigma^n(e)), \text{ and}
\end{align}
\end{minipage}}
\setlength\itemsep{0pt}
\item
\makebox[0pt][l]{
\begin{minipage}[b]{\textwidth}
\begin{align}\label{item:coincidence-with-infinite-system}
 \sL(X_{\sigma}) = \sL(\sigma).
\end{align}
\end{minipage}}
\end{itemize}

Thus, their class of substitutions is a proper subclass of the class of
 substitutions by our definition.
They showed the following:
\begin{thm}[{\cite[Theorem 1]{Damanik2006SubstitutionDynSysCharactOfLinearRepetAndApplications}}]\label{thm:DL}
Let $\sigma : A \to A^+$ be a substitution that satisfies
 \cref{item:infinite-char,item:all-char,item:coincidence-with-infinite-system}.
Then, the following statements are equivalent:
\enumrmb
\item
 There exists an $e \in A$ that satisfies
 \cref{item:infinite-char,item:all-char}
 and a $K > 0$ such that
 if\ \ $e w e \prec \sigma^n(e)$ with $e \not\in \Lett(w)$, then $\abs{w} < K$;
\item\label{item:DL2} $(X_{\sigma},T)$ is minimal; and
\item\label{item:DL3} $(X_{\sigma}, T)$ is linearly repetitive.
\enumrmn
\end{thm}

After the discussion by Cortez and Solomyak
 \cite{Cortez2011InvariantMeasuresForNonPrimitiveTilingSubstitutions}
 in the case of tile substitutions,
 Maloney and Rust in \cite[\S 1.2]{Maloney_Rust_2016BeyondPrimitivityForOneDimensionalSubstitutionSubshiftsAndTilingSpaces}
 called a word $u$ \textit{legal} if $u \in \sL(X_{\sigma})$.
They noted that, in \cite{Cortez2011InvariantMeasuresForNonPrimitiveTilingSubstitutions}, it is shown that a substitution is admissible
 if and only if every letter in $A$ is legal.
Thus, the above theorem is for every admissible substitution with a special
 letter $e$.

\begin{defn}\label{defn:atob}
Let $a,b \in A$.
We write $a \to b$ if there exists an $n \bpi$
 such that $b \in \Lett(\sigma^n(a))$.
\end{defn}

We note that for each $a \in A$, there exists at least one letter $b \in A$
 with $a \to b$.
Furthermore, it follows that if $a \in A_l$, then there exists a $b \in A_l$
 with $a \to b$.
We define as $A_{\circ} := \seb a \in A \mid a \to a \sen$
  and $A_{\circ,l} := A_{\circ} \cap A_l$.
The proofs of the following facts are left to the readers:
\begin{itemize}[label={},topsep=12pt]
\setlength\itemsep{0pt}
\item
\makebox[0pt][l]{
\begin{minipage}[b]{\textwidth}
\begin{align}\label{item:atob-btoc-implies-atoc}
 (a \to b \myand b \to c) \text{ implies } a \to c,
\end{align}
\end{minipage}}
\item
\makebox[0pt][l]{
\begin{minipage}[b]{\textwidth}
\begin{align}\label{item:AcircAl-is-not-empty}
 A_{\circ,l} \nekuu,
\end{align}
\end{minipage}}
\setlength\itemsep{0pt}
\item
\makebox[0pt][l]{
\begin{minipage}[b]{\textwidth}
\begin{align}\label{item:AcircAl-implies-sigmaaAcircAl}
 \text{every } a \in A_{\circ,l} \text{ satisfies }
 \Lett\left(\sigma(a)\right) \cap A_{\circ,l} \nekuu, \myand
\end{align}
\end{minipage}}
\item
\makebox[0pt][l]{
\begin{minipage}[b]{\textwidth}
\begin{align}\label{item:AcircAl-implies-sigmanaAcircAl}
 \text{for every } a \in A_{\circ,l} \myand n \bpi,
 \text{ it follows that } \Lett(\sigma^n(a)) \cap A_{\circ,l} \nekuu.
\end{align}
\end{minipage}}
\end{itemize}

\begin{defn}\label{defn:minimal}
A letter $a \in A_l$ is \textit{minimal in} $A_l$ if
 $a \to b$ with $b \in A_l$ implies $b \to a$.
We also define
 $A_{\min,l} := \seb a \in A_l \mid a \text{ is minimal in } A_l \sen$.
\end{defn}

\begin{lem}\label{lem:AlAcirc-includes-Aminl-and-nekuu}
It follows that $A_{\circ,l} \supseteq A_{\min,l} \nekuu$.
\end{lem}
\begin{proof}
It is evident that $A_l \supseteq A_{\min,l}$.
We have to show that $A_{\circ} \supseteq A_{\min,l}$.
Let $a \in A_{\min,l}$.
Then, there exists a $b \in A_l$ with $a \to b$.
By minimality in $A_l$, we obtain $b \to a$.
Thus, we obtain $a \to a$, i.e., $a \in A_{\circ}$.
To show that $A_{\min,l} \nekuu$,
 it is sufficient to note that $A_l \nekuu$ (our standing hypothesis for
 the substitutions) and $A_l$ is a finite set.
\end{proof}

Let $a \in A_l$ be a left-isolated letter.
Then, there exists an $n_{\minus}(a) \bpi$ such that
 $\sigma^{n_{\sminus}(a)}(a) =  l_s(a)\ a\ w_{\plus}(a)$ with
 $l_s(a) \in {A_s}^+$ and $w_{\plus}(a) \in A^*$.
Similarly, for a right-isolated letter $a \in A_l$,
 there exists an $n_{\plus}(a) \bpi$ such that
 $\sigma^{n_{\splus}(a)}(a) =  w_{\minus}(a)\ a\ r_s(a)$ with
 $r_s(a) \in {A_s}^+$ and $w_{\minus}(a) \in A^*$.
We note that $l_s(a)$ and $r_s(a)$ cannot be $\ew$.
On the other hand, $w_{\minus}(a)$ and $w_{\plus}(a)$
 may be $\ew$.
For a left-isolated letter $a \in A_l$ and $k \bpi$,
 we obtain the expression
 $\sigma^{n_{\sminus}(a)k}(a) =  l_s(k,a)\ a\ w_{\splus}(k,a)$ with
 $l_s(k,a) \in {A_s}^+$.
We obtain $\abs{l_s(k,a)} \to \infty$ as $k \to \infty$.
In the same way, for a right-isolated letter $a \in A_l$ and $k \bpi$,
 we obtain the expression 
 $\sigma^{n_{\splus}(a)k}(a) =  w_{\minus}(k,a)\ a\ r_s(k,a)$ with
 $r_s(k,a) \in {A_s}^+$, and $\abs{r_s(k,a)} \to \infty$ as $k \to \infty$.

\begin{lem}[{\cite[Lemma 2.8]{Maloney_Rust_2016BeyondPrimitivityForOneDimensionalSubstitutionSubshiftsAndTilingSpaces}}, \cite{Pansiot1984ComplexiteDesFacteursDesMotsInfinisEngendresParMorphismesIteres}]\label{lem:isolated-periodic}
Suppose that there exists a left- or right-isolated
 letter in $A_l$.
Then, $(X_{\sigma},T)$ contains a periodic point.
\end{lem}
\begin{proof}
Suppose that there exists a left-isolated letter $a \in A_l$.
We employ the above notation
 and obtain the expression
 $\sigma^{n_{\sminus}(a)}(a) =  l_s(a)\ a\ w_{\plus}(a)$ and also
 $\sigma^{n_{\sminus}(a)k}(a) =  l_s(k,a)\ a\ w_{\splus}(k,a)$ with
 $l_s(k,a) \in {A_s}^+$.
From these, we can derive the equation
\[ l_s(k,a) = \sigma^{k-1}(l_s(a))\ \sigma^{k-2}(l_s(a))\ \cdots\ l_s(a).\]
It is easy to see that the left half of $l_s(k,a)$ is periodic for 
 sufficiently large $k$.
Thus, $(X_{\sigma},T)$ contains a periodic point.
A similar argument shows the proof for the right-isolated case.
\end{proof}

The result of the above lemma is equivalent to stating that
 if a substitution subshift does not contain any periodic orbits, then
 it is tame.

\begin{rem}\label{rem:tameness}
Because the condition of tameness has been considered in preceding works
 (see \cite{Pansiot1984ComplexiteDesFacteursDesMotsInfinisEngendresParMorphismesIteres,Maloney_Rust_2016BeyondPrimitivityForOneDimensionalSubstitutionSubshiftsAndTilingSpaces}),
 even if we do not write explicitly,
 many of our results concerning tameness in this paper
 should be parts of them.
\end{rem}

\begin{rem}\label{rem:check-not-isolated}
We note that both left isolation and right isolation can be checked by finite calculations.
Therefore, the tameness can be checked by finite calculations.
\end{rem}

For each $a \in A_l$ and $n \bpi$, we write
\begin{align}\label{align:sigmana}
 \sigma^n(a) = l_s(n,a)\ & a_{\minus,n}\ w_s(n,a,1)\ a_{n,1}\ w_s(n,a,2)
 \ a_{n,2}\ w_s(n,a,3) \\
 \ & \hspace{2.8cm} \dotsb\  w_s(n,a,j(n,a))\ a_{\plus,n}\ r_s(n,a),\nonumber
\end{align}
in which
 $a_{\minus,n}, a_{\plus,n} \in A_l, a_{n,i} \in A_l$
 for $(1 \le i < j(n,a))$,
 $l_s(n,a), r_s(n,a) \in {A_s}^*,\ 
 w_s(n,a,i) \in {A_s}^*$ for $(1 \le i \le j(n,a))$.
We note that we permit $w_s(n,a,i) = \ew$ and also 
 $\sigma^n(a) = l_s(n,a)\ a_{\minus,n} (= a_{\plus,n})\ r_s(n,a)$.
For each $w_s(n,a,i) \in {A_s}^*$, it follows that
 $\abs{\sigma^m(w_s(n,a,i))} \le k_s\abs{w_s(n,a,i)}$
 for each $m \bpi$ (cf. \cref{defn:As-const}).

\begin{rem}\label{rem:tame-implies-head-letter}
The sequence $a_{\minus,n}$ $(n=1,2,\dotsc)$ is eventually periodic.
Therefore,
 there exists a letter $b \in A_l$ such that if we set $b_{\minus,0} = b$,
 then $b_{\minus,n}$ $(n=0,1,2,\dotsc)$ is periodic.
If $l_s(n,b) \in {A_s}^+$ for some $n \bpi$, then the substitution
 $\sigma$ is not tame.
Therefore, if a substitution $\sigma : A \to A^+$ is tame,
 then there exists a letter $a \in A_l$ and a $p \bpi$ such that 
 $l_s(n,a) = \ew$ for all $n \bpi$ and $\sigma^p(a) = a w(a,p)$
 for some $w(a,p) \in A^+$.
Thus, by cutting the nontame substitutions
 and taking $\sigma^p$ instead of $\sigma$ itself,
 we obtain the standing hypothesis that is assumed in many preceding studies,
 i.e., there exists a letter $a \in A_l$ that satisfies
 $\sigma(a) = a w$ with $w \in A^+$
 (see also \cref{prop:exists-periodic-point-of-sigma}).
\end{rem}

\begin{rem}
Suppose that a substitution $\sigma : A \to A^+$ is tame.
Then, $a_{\minus,n}$ and $a_{\plus,n}$ are eventually periodic
 as $n$ grows.
Thus, by the definition of tameness, we obtain
 $\lim_{n \to \infty}\abs{l_s(n,a)} < \infty$ and 
 $\lim_{n \to \infty}\abs{r_s(n,a)} < \infty$.
\end{rem}

\begin{defn}\label{defn:As-bi-sided-limit-in-non-isolated-case}
Suppose that a substitution $\sigma : A \to A^+$ is tame.
We define\\
 $e_{s,\minus} := \max_{n \bpi, a \in A_l}\abs{l_s(n,a)}$,
 $e_{s,\plus} := \max_{n \to \infty, a \in A_l}\abs{r_s(n,a)}$, and
 $e_s := \max \seb e_{s,\minus}, e_{s,\plus} \sen$.
\end{defn}

\begin{nota}
For $a \in A_l$ and $n \bpi$, we denote
\[\gap_s(n,a) := \max_{1 \le i \le j(n,a)} \abs{\sigma^{n_s}w(n,a,i)} \le
 \max_{1 \le i \le j(n,a)}k_s \abs{w(n,a,i)}, \myand\]
\[\gap_s(n) := \max_{a \in A_l}\gap_s(n,a).\]
We note that $\gap(n)$ is nondecreasing.
\end{nota}

The following result is a part of \cite{Maloney_Rust_2016BeyondPrimitivityForOneDimensionalSubstitutionSubshiftsAndTilingSpaces}.
It is also found in \cite{Pansiot1984ComplexiteDesFacteursDesMotsInfinisEngendresParMorphismesIteres} in the context of D0L systems.

\begin{lem}\label{lem:gap-by-As-bounded}
Suppose that $\sigma : A \to A^+$ is a tame substitution.
Then, we obtain\\
 $\lim_{n \to \infty}\gap_s(n) < \infty$.
\end{lem}
\begin{proof}
Take and fix an $a \in A_l$ and an $n \bpi$ arbitrarily.
In the calculation in \cref{align:sigmana}, consider
 $\sigma^m(a_{n,i}\ w(n,a,i) \ a_{n,i+1})$ for $1 \le i < j(n,a)$
 and an arbitrary $m \bpi$.
 We define $a_{n,j(n,a)} := a_{+,n}$.
We rewrite this as $\sigma^m(b\ w\ c)$ with $b = a_{n,i}, c = a_{n,i+1}$,
 and $w = w(n,a,i)$.
By \cref{item:not-isolated},
 we obtain that $b = a_{n,i}$ is not right-isolated and $c = a_{n,i+1}$ is not
 left-isolated.
Therefore, we can calculate the following:
\begin{align}
\begin{split}
 \sigma^m(b\ w\ c) & = \sigma^m(b)\ \sigma^m(w) \ \sigma^m(c)\\
 \ & = w(m,b)\ b_{\plus,m}\ r_s(m,b)\ \sigma^m(w)
 \ l_s(m,c)\ c_{\minus,m}\ w(m,c),
\end{split}
\end{align}
with some $b_{\plus,m}, c_{\minus,m} \in A_l$, $w(m,b), w(m,c) \in A^*$,
 and $r_s(m,b), l_s(m,c) \in {A_s}^*$\\
 with $\abs{r_s(m,b)}, \abs{l_s(m,c)} \le e_s$.
Each gap in $\sigma^m(a_{n,i}\ w(n,a,i) \ a_{n,i+1})$ appears as
\[ r_s(m,b)\ \sigma^m(w)\ l_s(m,c)\]
 or appears in $\sigma^m(b)$ or $\sigma^m(c)$.
As $m$ becomes larger, the length of the gap
 $r_s(m,b)\ \sigma^m(w)\ l_s(m,c)$\quad may become larger;
 however, it has an upper bound that is estimated as
 $e_s + k_s \cdot \abs{w} +e_s$
 (cf. \cref{defn:As-bi-sided-limit-in-non-isolated-case,defn:As-const}).
Thus, the length of the gap has an upper bound $\le k_s \cdot \gap_s(n) +2e_s$.
We consider the case in which $m = n$.
The length of the new gaps in $\sigma^n(b)$ or $\sigma^n(c)$ is bound
 by $\gap(n)$.
Thus, we obtain $\gap_s(2n) \le k_s \cdot \gap_s(n) + 2e_s$.
We consider the case in which $m = 2n$.
Then, in $\sigma^{3n}(a)$, the gaps that have appeared
 in $\sigma^{2n}(a)$ shall be 
 bound as $\gap_s(2n) \le k_s \cdot \gap_s(n) + 2e_s$.
The new gaps in $\sigma^{3n}(a)$ are in $\sigma^n(a')$ $(a' \in A_l)$.
Thus, we obtain $\gap_s(3n) \le k_s \cdot \gap_s(n) + 2e_s$.
In this way,
 we obtain $\gap_s(kn) \le k_s \cdot \gap_s(n) + 2e_s$ for all $k \bpi$.
Because $\gap_s(i)$ is nondecreasing,
 we obtain $\gap_s(i) \le k_s \cdot \gap_s(n) + 2e_s$ for all $i \bpi$.
This completes the proof.
\end{proof}

We continue to study about tameness.
Let $\sigma : A \to A^+$ be a substitution.
For each $x \in A^{\Z}$, we write $x = x_{\minus}. x_{\plus}$
 $(x_{\minus} \in A^{(-\infty,-1]}, x_{\plus} \in A^{[0,\infty)})$
 such that $x_{\minus}(i) = x(i)$ ($i < 0$) and $x_{\plus}(i) = x(i)$
 ($0 \le i$).
In the introduction, we have already defined
 a map $\sigma : A^{\Z} \to A^{\Z}$ by
 $\sigma(x)
 := \dotsb \sigma(x(-2)) \sigma(x(-1)). \sigma(x(0)) \sigma(x(1)) \dotsb$.
We would like to show that $\sigma(X_{\sigma}) \subseteq X_{\sigma}$.
For every $x \in X_{\sigma}$ and $s < t$,
 there exists an $a \in A_l$ such that
 $x[s,t] \in \sL(a,\infty)$.
Then, it follows that for all $s < t$, there exists an $a \in A_l$ with
 $\sigma(x[s,t]) \in \sL(a,\infty) \subseteq \sL(\sigma)$,
 i.e., $\sigma(x) \in X_{\sigma}$.
Therefore, we obtain $\sigma(X_{\sigma}) \subseteq X_{\sigma}$.
Thus, we obtain a map $\sigma : X_{\sigma} \to X_{\sigma}$.
In \cite{Maloney_Rust_2016BeyondPrimitivityForOneDimensionalSubstitutionSubshiftsAndTilingSpaces}, to clarify the notion of tameness,
 they described that most of the possible pathological
 behaviors of nonminimal substitutions cannot occur in the presence of tameness.
The next proposition shows one of the examples of how tameness works.

\begin{prop}\label{prop:exists-periodic-point-of-sigma}
Let $\sigma : A \to A^+$ be a tame substitution.
There exist a positive integer $p$,
 letters $b_0,c_0 \in A_l$, and a word $w_0 \in {A_s}^+$
 that satisfy $\sigma^p(b_0) = w_- b_0$, $\sigma^p(c_0) = c_0 w_+$,
 $\sigma^p(w_0) = w_0$,
 and $b_0 w_0 c_0 \in \sL(\sigma)$.
In particular, there exist an $x \in X_{\sigma}$, $s < 0$,
 and a $p \bpi$ such that $x[s,-1] = b_0 w_0$, $x(0) =  c_0$, and
 $\sigma^p(x) = x$.
\end{prop}
\begin{proof}
Let $a \in A_l$.
Because $\abs{\sigma^i(a)} \to \infty$ as $i \to \infty$,
 \cref{lem:gap-by-As-bounded} implies that
 there exist an $N > 0$ and a word $b w c \prec \sigma^N(a)$ with
 $b,c \in A_l$ and $w \in {A_s}^*$.
For each $m \bpi$, we can write
\begin{align}
 \sigma^m(b\ w\ c) & = \sigma^m(b)\ \sigma^m(w) \ \sigma^m(c)\\
 \ & = w(m,b)\ b_{\plus,m}\ r_s(m,b)\ \sigma^m(w)
 \ l_s(m,c)\ c_{\minus,m}\ w(m,c)\nonumber
\end{align}
for some $b_{\plus,m}, c_{\minus,m} \in A_l$, $w(m,b), w(m,c) \in A^*$,
 and $r_s(m,b), l_s(m,c) \in {A_s}^*$ with\\
 $\abs{r_s(m,b)},\ \abs{l_s(m,c)} \le e_s$.
The sequence of pairs $(b_{\plus,m},c_{\minus,m})$ $(m \bpi)$
 is eventually periodic, and
 $r_s(m,b) \sigma^m(w)  l_s(m,c)$ is also eventually periodic.
Therefore, there exists an $l \bpi$ and a $p \bpi$ such that
 if we write $b_0 := b_{\plus,l}$, $c_0 := c_{\minus,l}$ and
 $w_0 := r_s(l,b) \sigma^l(w)  l_s(l,c)$, then
 $\sigma^p(b_0 w_0 c_0) = \sigma^p(b_0)\ \sigma^p(w_0) \sigma^p(c_0)
 = w_- b_0 \sigma^p(w_0) c_0 w_+ = w_- b_0 w_0 c_0 w_+$
 with $w_-, w_+ \in A^+$.
By applying $\sigma^p$ infinitely on $b_0 w_0 . c_0$,
 we can obtain a fixed point $\sigma^p(x) = x \in X_{\sigma}$
 such that $x(0) = c_0$.
This completes the proof.
\end{proof}

In addition to the tameness, $l$-primitivity 
 plays a central role in this paper.

\begin{rem}\label{rem:primitive-Al}
Let $\sigma : A \to A^+$ be a substitution that is $l$-primitive.
By definition, there exists an $n \bpi$ such that
 for every $a,b \in A_l$, it follows that $a$ is a letter in $\sigma^n(b)$.
Then, it follows that for every $a \in A_l$, there exists a $c \in A_l$
 such that $a \in \Lett(\sigma(c))$.
Thus, for every $m \ge n$, it follows that $\abs{\sigma^m(b)}_a \ge 1$
 for arbitrary $a, b \in A_l$.
\end{rem}

\begin{lem}\label{lem:every-letter-in-Al-in-bounded-gaps}
Let $\sigma : A \to A^+$ be a substitution.
Suppose that $\sigma$ is tame and $l$-primitive.
Let $x \in X_{\sigma}$ be a periodic point of
 $\sigma :X_{\sigma} \to X_{\sigma}$
 and $p \bpi$ such that $\sigma^p(x) = x$.
Then, in the sequence $x$, every letter $a \in A_l$ appears in bounded gaps.
\end{lem}
\begin{proof}
By assumption, there exists an $n \bpi$ such that
 $\abs{\sigma^n(b)}_a \ge 1$ for every $a, b \in A_l$.
By \cref{rem:primitive-Al}, for every $m \ge n$ and $a \in A_l$,
 it follows that $\Lett(\sigma^m(a)) \supseteq A_l$.
Because $\sigma^p(x) = x$, it follows that
 $\sigma^{kp}(x) = x$ for every $k \bpi$.
We take and fix $k \bpi$ such that $p' := kp \ge n$.
Let us define $l := \max_{a \in A}\abs{\sigma^{p'}(a)}$.
Because $\sigma^{p'}(x) = x$, we can write as 
 $x = \cdots \sigma^{p'}(x(-1)). \sigma^{p'}(x(0)) \sigma^{p'}(x(1)) \cdots$.
If $x(i) \in A_l$ ($i \bi$), then $\Lett(\sigma^{p'}(x(i))) \supseteq A_l$. 
On the other hand, by \cref{lem:gap-by-As-bounded}, we can obtain
 $d = \lim_{n \to \infty}\gap_s(n) < \infty$.
Thus, in every block of $x$ with length $d + 2l$, there exists
 an appearance of $\sigma^{p'}(a)$ for some $a \in A_l$.
Therefore, in every block of $x$ with length $d + 2l$, every letter in $A_l$
 appears.
This completes the proof.
\end{proof}

Here, we would like to show that if our substitution $\sigma$
 is tame and $l$-primitive, then conditions
 \cref{item:infinite-char,item:all-char,item:coincidence-with-infinite-system}
 are satisfied.
Let $\sigma : A \to A^+$ be our substitution that is tame and $l$-primitive.
Because we have imposed the condition $A_l \nekuu$ on our substitutions,
 condition \cref{item:infinite-char} is evidently satisfied.
Take an $e \in A_l$ arbitrarily.
By the definition of $l$-primitivity, we obtain
 $A_l \subseteq  \bigcup_{n \bpi} \Lett( \sigma^n(e))$.
By this, because we have also imposed the condition $A_\text{iso} \iskuu$,
 condition \cref{item:all-char} follows.
The next proposition shows that condition 
 \cref{item:coincidence-with-infinite-system} is satisfied.
It is evident from the definitions
 that $\sL(X_{\sigma}) \subseteq \sL(\sigma)$.
It is easy to see that if every letter in $A$ is legal,
 then, by applying $\sigma : X_{\sigma} \to X_{\sigma}$ successively,
 every word $\sigma^n(a)$ $(n \bpi)$ is legal.
Thus, if we can show that every letter in $A$ is legal,
 then we can obtain $\sL(X_{\sigma}) \supseteq \sL(\sigma)$.
\begin{prop}\label{prop:admissible}
Let $\sigma : A \to A^+$ be a tame and $l$-primitive substitution.
Then, it follows that every letter in $A$ is legal.
Particularly, it follows that $\sL(X_{\sigma}) = \sL(\sigma)$.
\end{prop}
\begin{proof}
Let $a \in A$.
We need to show that $a$ is legal.
By \cref{prop:exists-periodic-point-of-sigma},
 $\sigma : X_{\sigma} \to X_{\sigma}$ has a periodic point
 $x \in X_{\sigma}$ with a period $p$ such that $x(0) \in A_l$.
By $l$-primitivity, all of the letters in $A_l$ appear in $x$.
It follows that $A = \bigcup_{i = 0,1,\dotsc,p-1}\Lett(\sigma^i(x))
 \subset \sL(X_{\sigma})$.
Henceforth, all letters in $A$ are legal.
\end{proof}

\begin{rem}\label{rem:calculation-primitive-Al}
For every substitution $\sigma : A \to A^+$,
 we can calculate
 $F_n(a) := \Lett(\sigma^n(a)) \subseteq A$ for each $a \in A$ and $n \bpi$.
The sequence $F_n(a)$ is eventually periodic.
Therefore, we can stop calculations by finite steps.
A substitution $\sigma$ is $l$-primitive
 if and only if there exists an $n \bpi$ such that
 $F_n(a) \supseteq  A_l$ for all $a \in A_l$.
Thus, we can check $l$-primitivity by finite calculations.
\end{rem}

\begin{thm}\label{thm:primitive-with-non-isolated-implies-minimal}
Suppose that $\sigma : A \to A^+$ is a tame and $l$-primitive substitution.
Then, the substitution subshift $(X_{\sigma},T)$ is minimal.
\end{thm}
\begin{proof}
Let $x \in X_{\sigma}$ be a periodic point of
 $\sigma :X_{\sigma} \to X_{\sigma}$
 and $p \bpi$ be its period, i.e., $\sigma^p(x) = x$.
By \cref{prop:minimal-subshifts}, to show the minimality of
 $(X_{\sigma},T)$, we have to show that 
 each finite block $v \prec x$ appears in the bounded gaps in $x$.
Let $v \prec x$ be a finite block.
Because $\sL(X_{\sigma}) \subseteq \sL(\sigma)$ and $\sigma$ is $l$-primitive,
 there exists an $n > 0$ such that $v \prec \sigma^{np}(a)$.
By \cref{lem:every-letter-in-Al-in-bounded-gaps}, $a$ appears in bounded gaps.
It follows that $\sigma^{np}(a)$ appears in the bounded gaps in $x$.
Thus, $v$ appears in the bounded gaps in $x$.
By \cref{prop:minimal-subshifts}, it remains to show that
 the orbit of $x$ is dense in $X_{\sigma}$.
Let $y \in X_{\sigma}$ and $s < t$.
Then, it follows that $y[s,t] \in \sL(\sigma)$.
Now, as in the first part of this proof,
 it follows that $y[s,t]$ appears in $x$ infinitely
 many times (in bounded gaps).
\end{proof}

We consider the next condition on a substitution
 $\sigma : A \to A^+$:
\begin{align}\label{item:minimal-not-periodic}
 \text{The resulting } (X_{\sigma},T) \text{ is minimal and is not a single periodic orbit.}
\end{align}

By \cref{lem:isolated-periodic},
 if a substitution $\sigma$
 satisfies \cref{item:minimal-not-periodic},
 then $\sigma$ is tame.
Hereon, we are going to show that if a substitution $\sigma$ satisfies
 condition \cref{item:minimal-not-periodic}, then there exists
 a nonempty subset $B \subset A$ such that $\sigma(B) \subseteq B^+$
 and the restriction $\sigma|_B$ is also $l$-primitive (the last half
 of \cref{thm:main}).

\begin{defn}\label{defn:decompose-minimal-letters}
We define a relation on $A_{\min,l}$ by
 $a \simminl b$ $(a,b \in A_{\min,l})$ if $a \to b$.
\end{defn}
By the definition of $A_{\min,l}$, it follows that
 ($a \simminl b$ implies $b \simminl a$).
By definition, we obtain that ($a \simminl b$ and $b \simminl c$)
 implies $a \simminl c$.
For each $a \in A_{\min,l}$, there exists a $b \in A_{\min,l}$ with
 $a \to b$.
Therefore,
 we obtain $a \simminl a$ for every $(a \in A_{\min,l})$.
Thus, $\simminl$ is an equivalence relation.
We decompose $A_{\min,l} = A_{\min,l,1} \cup A_{\min,l,2} \cup \dotsb \cup A_{\min,l,r(\sigma)}$ owing to this equivalence relation.

\begin{lem}\label{lem:r=1}
If we assume condition \cref{item:minimal-not-periodic},
 then $r(\sigma) = 1$.
\end{lem}
\begin{proof}
Suppose that $r(\sigma) > 1$, and
 let $A' = A_{\min,l,i}$ and $A'' = A_{\min,l,j}$
 with $i \ne j$.
Let $b' \in  A'$ and $b'' \in A''$.
Let $B' := A' \cup \bigcup_{n \bpi}\Lett(\sigma^n(b'))$ and 
 $B'' := A'' \cup \bigcup_{n \bpi}\Lett(\sigma^n(b''))$.
Because $i \ne j$,
 it follows that $B' \cap B'' \cap A_l = \kuu$.
Furthermore, it follows that $\sigma(B') \subset B'^+$ and
 $\sigma(B'') \subset B''^+$.
Because $\sigma$ is tame,
 both $\sigma|_{B'}$ and $\sigma|_{B''}$ are tame.
It follows that both $X_{\sigma|_{B'}}$ and $X_{\sigma|_{B''}}$ are nonempty
 subsets of $X_{\sigma}$.
It also follows that, by \cref{lem:gap-by-As-bounded},
 at least one letter in $B' \cap A_l$ and in $B'' \cap A_l$ have to 
 appear in each element of $X_{\sigma|_{B'}}$ and $X_{\sigma|_{B''}}$.
Thus, we obtain $X_{\sigma|_{B'}} \cap X_{\sigma|_{B''}} \iskuu$,
 a contradiction.
\end{proof}

Owing to the above lemma, we assume hereon that $r(\sigma) = 1$.
Thus, we can assume that for each $a, b \in A_{\min,l}$,
 we obtain both $a \to b$ and $b \to a$.
We define $B := \bigcup_{a \in A_{\min,l}, n \bpi}\Lett(\sigma^n(a))$.
Note that because every $a \in A_{\min,l}$ satisfies $a \to a$,
 it follows that $B \supseteq A_{\min,l}$.
Further, it is easy to see that $B \cap A_l = A_{\min,l}$.
We write as $B_l := B \cap A_l = A_{\min,l}$ and $B_s := B \setminus B_l$.
As we have already seen in the proof of the above lemma,
 we can consider a substitution $\sigma|_B : B \to B^+$.
We denote $\sS(B_l) := \seb F \mid F \subseteq B_l \mywith F \nekuu\sen$.
Then, it follows that for all $F \in \sS(B_l)$, we can define
 $s(F) \in \sS(B_l)$
 by $s(F) := \bigcup_{a \in F}\left(\Lett(\sigma(a)) \cap B_l \right)$.
Henceforth, we can define a map $s :\sS(B_l) \to \sS(B_l)$.
We note that for each $a \in B_l$ and $n \bpi$, it follows that
 $s^n(\seb a \sen) = \Lett(\sigma^n(a)) \cap B_l$.
For each $a \in B_l$, the sequence
 $s^n(\seb a \sen)$ $(n \bpi)$ is eventually periodic.
Thus, we denote the least eventual period as $p(a)$.

\begin{lem}\label{lem:period-not-1-implies-not-minimal}
Suppose that \cref{item:minimal-not-periodic} is satisfied with $\sigma$.
Let $B$ be as above.
Then, for every $a \in B_l$, we obtain $p(a) = 1$.
\end{lem}
\begin{proof}
As we have noted just after condition \cref{item:minimal-not-periodic},
 $\sigma$ is tame.
Fix $a \in B_l$.
Suppose that $a \in B_l$ satisfies $p(a) > 1$.
There exists an $N(a) \bpi$ and an $F(a) \in \sS(B_l)$ such that
 $s^{N(a)}(\seb a \sen) = F(a)$ and the sequence
 $s^i(F(a))$ $(i = 0,1,2,\dotsc)$
 is periodic with the least period $p(a)$.
Because $p(a)$ is the least period,
 it follows that $s^i(F(a))$ $(0 \le i < p(a))$
 are mutually distinct.
We define $F := \bigcup_{n \bpi, b \in F(a)}\Lett(\sigma^{p(a)n}(b))$.
Then, it follows that $F \cap B_l = F(a)$.
It follows that $F \subset B$ and $s^{p(a)}(F) = F$.
Thus, the substitution $\tau := \sigma^{p(a)}|_F : F \to F^+$
 is well-defined.
Because $\sigma$ is tame, $\tau$ is also tame.
By \cref{prop:exists-periodic-point-of-sigma},
 we obtain a periodic point $x \in X_{\tau}$ under the action
 $\tau : X_{\tau} \to X_{\tau}$ such that
 $\Lett(x) = F(a)$.
In a similar manner, defining
 $F_1 := \bigcup_{n \bpi, b \in s(F(a))}\Lett(\sigma^{p(a)n}(b))$,
 we obtain $F_1 \cap B_l = s(F(a))$.
It follows that $F_1 \subset B$ and $s^{p(a)}(F_1) = F_1$.
Thus, the substitution $\tau_1 := \sigma^{p(a)}|_{F_1} : F_1 \to {F_1}^+$
 is well-defined and tame.
Similarly, we obtain a periodic point $x' \in X_{\tau_1}$ under the action
 $\tau_1 : X_{\tau_1} \to X_{\tau_1}$ such that
 $\Lett(x') = s(F(a))$.
Thus, it follows that $\Lett(x') \ne \Lett(x)$.
Thus, we can conclude that both $(X_{\tau},T)$ and $(X_{\tau_1},T)$
 are proper subshifts of $X_{\sigma}$, contradicting the minimality
 of $X_{\sigma}$.
\end{proof}

\begin{prop}\label{prop:minimal-not-periodic-implies-expression-by-primitive}
Let $\sigma : A \to A^+$ be a substitution.
Suppose that \cref{item:minimal-not-periodic} is satisfied.
Then, $\sigma$ is tame, and
 there exists a $B \subseteq A$ such that
 $B_l := B \cap A_l = A_{\min,l}$ and
 $\sigma|_B$ defines a substitution $\sigma|_B : B \to B^+$
 that is tame and $l$-primitive and
 $(X_{\sigma|_B},T) = (X_{\sigma},T)$.
\end{prop}
\begin{proof}
Let $\sigma : A \to A^+$ be a substitution
 that satisfies \cref{item:minimal-not-periodic}.
We have already shown that $\sigma$ is tame and defined
 $B \subseteq A$ such that $B_l := B \cap A_l = A_{\min,l}$.
Furthermore, $\sigma|_B$ defines a substitution $\sigma|_B : B \to B^+$.
The subshift $(X_{\sigma|_B},T)$ is a subshift of $(X_{\sigma},T)$.
Thus, condition \cref{item:minimal-not-periodic} implies that
 $(X_{\sigma|_B},T) = (X_{\sigma},T)$.
Henceforth, we only need to show the $l$-primitivity of $\sigma|_B$.
However, this is a direct consequence
 of \cref{lem:period-not-1-implies-not-minimal}.
This concludes the proof.
\end{proof}

\noindent \textit{Proof of \cref{thm:main}}. 
Let $\sigma : A \to A^+$ be a substitution.
Suppose that $\sigma$ is tame and $l$-primitive.
Then, by \cref{thm:primitive-with-non-isolated-implies-minimal},
 we obtain that $(X_{\sigma},T)$ is minimal.
Conversely,
 suppose that $(X_{\sigma},T)$ is minimal and is not a single periodic orbit.
Then, by \cref{prop:minimal-not-periodic-implies-expression-by-primitive},
 $\sigma$ is tame, and 
 there exists a $B \subseteq A$ and the restriction $\sigma|_B : B \to B^+$
 such that $\sigma|_B$ is tame and $l$-primitive
 and $(X_{\sigma},T) = (X_{\sigma|_B},T)$.
\qed

\vspace{4mm}

It is an easy exercise to show that 
 a subshift that is a single periodic orbit is a substitution minimal
 subshift generated by a primitive substitution.

\vspace{3mm}

\noindent \textit{Proof of \cref{thm:main2}}.
Suppose that $(X_{\sigma},T)$ is a substitution minimal subshift
 that is a single periodic orbit.
Then, by the above fact, it follows that $(X_{\sigma},T)$ is
 a substitution subshift that is generated by a substitution
 $\sigma : A \to A^+$ that is primitive---in particular,
 tame and $l$-primitive.
Suppose that $(X_{\sigma},T)$ is a substitution minimal subshift
 that is not a single periodic orbit.
Then, by the converse part of
 \cref{thm:main}, $\sigma$ is tame, and there exists a $B \subseteq A$
 and the restriction $\sigma|_B : B \to B^+$
 such that $\sigma|_B$ is tame and $l$-primitive
 and $(X_{\sigma},T) = (X_{\sigma|_B},T)$.
Thus, $(X_{\sigma},T)$ can be considered to be a substitution subshift
 generated by a tame and $l$-primitive substitution.
On the other hand, suppose that $\sigma$ is tame and $l$-primitive.
Then, by the first part of \cref{thm:main},
 it follows that $(X_{\sigma},T)$ is minimal.
\qed

\begin{cor}[{Part of {\cite[Theorem 1]{Damanik2006SubstitutionDynSysCharactOfLinearRepetAndApplications}}}]\label{cor:DL_my_way}
Let $\sigma : A \to A^+$ be a substitution.
Then, the substitution subshift $(X_{\sigma},T)$ is minimal if and only if
 it is linearly repetitive.
\end{cor}
\begin{proof}
By \cref{prop:minimal-subshifts}, linearly repetitive substitution subshifts
 are minimal.
We have to show the converse.
Let $\sigma : A \to A^+$ be a substitution such that 
 $(X_{\sigma},T)$ is minimal.
By \cref{thm:main2},
 we assume that $\sigma$ is tame and $l$-primitive.
Then, by \cref{prop:admissible},
 we can conclude that $\sL(X_{\sigma}) = \sL(\sigma)$.
Thus, we can apply 
 \cite[Theorem 1]{Damanik2006SubstitutionDynSysCharactOfLinearRepetAndApplications} itself to conclude that $(X_{\sigma},T)$ is linearly repetitive.
\end{proof}

\vspace{3mm}

As stated in the introduction,
 we could not give any examples of substitutions
 that generate minimal subshifts
 that cannot be generated directly by a primitive substitution.
Finally, we have not succeeded in providing a concrete
 algorithm for a given substitution to determine
 whether or not the substitution subshift is minimal.

\vspace{3mm}

\noindent
\textsc{Acknowledgments:}
We are most thankful for the many detailed suggestions
 provided by the anonymous referee(s) of
 our previous submission.
We would like to thank Editage (www.editage.jp)
 for providing English-language editing services.
This work was supported by JSPS KAKENHI Grant Number 16K05185.

\providecommand{\bysame}{\leavevmode\hbox to3em{\hrulefill}\thinspace}
\providecommand{\MR}{\relax\ifhmode\unskip\space\fi MR }
\providecommand{\MRhref}[2]{%
  \href{http://www.ams.org/mathscinet-getitem?mr=#1}{#2}
}
\providecommand{\href}[2]{#2}

\end{document}